\documentclass[11pt]{amsart}

\usepackage{amsthm}
\usepackage{amssymb,amsmath,amsfonts,microtype}
\usepackage{graphicx}
\usepackage{pdfsync}
\usepackage{parskip}

\usepackage[margin=1in]{geometry}
\usepackage{amssymb}
\usepackage[hidelinks]{hyperref}
\usepackage[capitalize]{cleveref}
\usepackage{microtype}
\usepackage{esint}

\newtheorem{thm}{Theorem}[section]

\newtheorem{lem}{Lemma}[section]
\newtheorem{prop}{Proposition}[section]

\newtheorem*{thm1'}{Theorem 1'}

\newtheorem*{thm2'}{Theorem 2'}

\newtheorem*{thmB'}{Theorem B$^\prime$}
\newtheorem*{thmA'}{Theorem A$^\prime$}
\newtheorem*{propnA'}{Proposition A$^\prime$}
\newtheorem*{propnB'}{Proposition B$^\prime$}
\theoremstyle{remark}

\theoremstyle{definition}

\theoremstyle{remark}

\newcommand{\Z}{\mathbb{Z}}
\newcommand{\N}{\mathbb{N}}
\newcommand{\F}{\mathbb{F}}
\newcommand{\PP}{\mathbb{P}}

\newcommand{\C}{\mathbb{C}}

\newcommand{\FF}{\mathcal{F}}
\newcommand{\RR}{\mathcal{R}}

\newcommand{\NN}{\mathcal{N}}
\newcommand{\HH}{\mathcal{H}}

\newcommand{\GG}{\mathcal{G}}

\newcommand{\ux}{\mathbf{x}}

\newcommand{\uz}{\mathbf{z}}

\newcommand{\ur}{\mathbf{r}}

\newcommand{\la}{\lambda}

\newcommand{\ga}{\gamma}

\newcommand{\subs}{\subseteq}

\numberwithin{equation}{section}
\newcommand{\eq}{\begin{equation}
		\newcommand{\ee}{\end{equation}}}

\title{On almost primes solutions to forms of odd degrees in many variables}

\author{\'Akos Magyar}
\thanks{The author was partially supported by grant \#854813 by the Simons Foundation.}

\begin{document}

	\begin{abstract} Let $\FF=\{f_1,\ldots,f_R\}$ be a family of forms of odd degrees at most $d$ in $s$ variables. We study the solutions to the system $f_1(\ux)=\ldots=f_R(\ux)=0$ of the form $x_i=y_ip_i$ with $|y_i|\leq Y_\FF$ and $p_i$ being a prime for all $i\in [s]$ inside the box $[-N,N]^s$, for large $N$. We show that if the number of variables $s$ is sufficiently large with respect to the parameters $R$ and $d$, then there are at least $C_\FF N^{s-D}/(\log\,N)^s$ such solutions for some constants $C_\FF>0$ and $D\in\N$, with $D$ depending only on the initial parameters $R$ and $d$.
	\end{abstract} 
	
	\maketitle
	
	\section{Introduction.} Analytical approaches to study integer solutions to diophantine systems in many variables have been initiated by Davenport \cite{DavenportCubic16} and Birch \cite{BirchForms}, and has been an area of interest since then. More recently it has been observed that under suitable rank and local conditions diophantine systems also admit solutions among the primes, and in fact asymptotic formulae may be obtained for the number of prime solutions \cite{CookMagyarPrimes,YamagishiPrimesDifferingDegrees,liu2023forms,liu2024forms}.  
	
	In a related direction it has been shown by Birch \cite{BirchOddDegree} that if the number of variables $s$ is sufficiently large with respect to the degrees and number of forms then there exists a non-trivial integer solution as long as all forms have odd degrees, without any further assumptions. Birch's approach was not analytical, it was based on a procedure to find subspaces on which the system take simpler forms. In 1985 Schmidt \cite{SchmidtDensity} has substantially strengthened Birch's result, he has shown that a system of $\mathcal{F}=\{f_1,\ldots,f_R\}$ consisting of odd degree integral forms has $\gg_\FF N^{s-D}$ integral solutions inside a large box $[-N,N]^s$, where $D$ is a constant depending only on the parameters $d$ and $r$. He introduced a new notion of rank, now is referred to as the Schmidt invariant or "strength" \cite{SchmidtDensity,baily2024strength}, and exploited its recursive nature to transform the system step by step until it has large enough rank so that the circle method applies to count integer solutions.
	
	In \cite{CookMagyarPrimes} Cook and author have made Schmidt's regularization process more explicit in the form a of a regularity lemma type decomposition of a general system of forms, and used to count primes solutions to high rank diophantine systems satisfying appropriate local conditions. In \cite{YamagishiPrimesDifferingDegrees} Yamagishi has extended the method of \cite{CookMagyarPrimes} to systems consisting of forms of different degrees. In later works \cite{yamagishi2022diophantine,liu2023forms,liu2024forms} it was observed the regularization procedure is not necessary to obtain asymptotics for the number of solutions to non-singular systems in prime variables.

	It is our aim to extend Schmidt's result, to show that systems of odd degree forms in sufficiently large number of variables admit many solutions $\ux=(x_1,\dots,x_s)$, where each coordinate $x_i$ can be written in the form $x_i=y_ip_i$ with $y_i$ being a bounded natural number and $p_i$ a prime number, with the exception of a bounded number of coordinates which may be zero. In this generality this latter exceptional set of coordinates is necessary as the system may contain equations of the form $x_i=0$, or linear forms (or even odd powers of linear forms) whose linear combination lead to such equations forcing some variables to be equal to 0.
	
	\newpage

	\section{Main results.} Let $\FF=(f_1,\dots,f_R)$ be a system consisting of odd-degree forms, where by a ``form" we mean a homogeneous polynomial $f\in\Z[\ux]$, $(\ux=(x_1,\ldots,x_s))$. For any odd number $1\leq l\leq d$, let $\FF^{(l)}$ consist of the forms of $\FF$ of degree $l$. Set $r_l=|\FF^{(l)}|$, $\mathbf{r}=(r_d,\ldots,r_1)$ and $R:=r_1+\ldots +r_d$.
	
	For a given parameter $Y>0$ we will define an almost-prime solution set. Let $\NN_\FF(N,Y)$ denote the number of integer solutions $\ux\in [-N,N]^s$ to the system $f_1(\ux)=\ldots =f_R(\ux)=0$ such that every coordinate takes the form $x_i=y_ip_i$, where $p_i\in\PP$ is a prime and $y_i\in\Z$ satisfies $0\leq |y_i|\leq Y$. Our first result establishes a lower bound on $\NN_\FF(N,Y)$  provided the number of variables is sufficiently large.
	
	\begin{thm}\label{thm2.1} There exist positive integers $s_0=s_0(R,d)$ and $D=D(R,d)$ such that for any system of odd-degree forms $\FF$ in $s\geq s_0$ variables the following holds. There exist constants $Y_\FF, c_\FF>0$ such that 
		\eq\label{1.1}
		\NN_\FF(N,Y_\FF) \geq c_\FF\, \frac{N^{s-D}}{(\log\,N)^s}
		\ee
		for sufficiently large $N$.
	\end{thm}
	
	Note that we allow $y_i=0$. However the number of integer solutions $x_i=0$ for all $i\in I$ for an index set $I$ of size $|I|>D$ is at most $N^{s-D-1}< c_\FF N^{s-D}/(\log\,N)^s$ for sufficiently large $N$. Thus the number of solutions in which at most $D$ coordinates are zero still satisfies the lower bound $c_\FF N^{s-D}/(\log\,N)^s$. Our method leads to a slightly stronger result allowing a small exceptional set of coordinates. Let $\NN_{\FF,J}(N,Y)$ denote the number of solutions where $x_i=y_ip_i$ with $1\leq |y_i|\leq Y$ for all $i\notin J$.
	
	\begin{thm}\label{thm2.2} There exist positive integers $s_0=s_0(R,d)$, $D=D(R,d)$ and $J_0=J_0(R,d)$ such that for any system of odd-degree forms $\FF$ in $s\geq s_0$ variables the following holds. There exists an index set $J\subs [s]$ of size $|J|\leq J_0$ and constants $Y_\FF, c_\FF>0$ for which
		\eq\label{1.2}
		\NN_{\FF,J}(N,Y_\FF) \geq c_\FF\, \frac{N^{s-D}}{(\log\,N)^s}
		\ee 
		for sufficiently large $N$.
	\end{thm}
	
	It is clear that one may need to set some variables equal to zero, e.g the system can contain equations $x_i=0$. The fact that we cannot have prime solutions even after setting a bounded number of variables to zero is necessary, consider for example the equation: 
	\[
	x_1-4x_2+\ldots +(-1)^{s-1}4^{s-1}x_s=0.
	\]
	If we look for solutions that are either primes or zero, then the only such solution is $x_1=\ldots=x_s=0$. However of we scale the coefficients with power of 4 we can reduce the equation to $p_1-p_2\ldots+(-1)^{s-1}p_s=0$ which have many prime solutions.
	To count almost-prime solutions for general system of odd-degree forms, our strategy proceeds in three main stages: regularization, finding non-singular local solutions and scaling. A system of odd-degree forms both have to be regular that is of sufficiently high rank and have non-vanishing reduced local factors, in order to have many solutions consisting of primes \cite{CookMagyarPrimes,yamagishi2022diophantine,liu2024forms}
	
	Thus first we apply a regularity decomposition to partition the solution space and extract an auxiliary system that is of sufficiently high rank. To guarantee this for the linear part of the system we need to set a bounded number of variables to be equal to zero. This yields a high rank system in $s'$ variables, including the linear forms of the system.
	
	Once this high-rank system is obtained, we use some facts from analysis and geometry to show the existence of non-singular $p$-adic solutions for a finite set exceptional primes $p$. With the aid of these non-singular solutions we carefully select a scaling vector $\mathbf{y}=(y_1,\ldots,y_{s'})$ and introduce the almost-prime constraint by restricting the variables to the form $x_i = y_i q_i$ yielding a new system of equations $\mathcal{F}_{\mathbf{y}}(\mathbf{q}) = 0$ in the prime variables $\mathbf{q} = (q_1, \ldots, q_{s'})$.
	
	The final system will have non-vanishing reduced local factors for all primes, and will have sufficiently large rank to utilize the Hardy-Littlewood circle method to estimate the number of prime solutions $\mathbf{q}$ providing the lower bounds stated in Theorems 2.1 and 2.2.

	\section{The main pillars of the proof.}
	
	We describe here the three main ``pillars" of the proof starting with the regularity decomposition of systems of integral forms. This procedure originates in the 1985 paper of Schmidt \cite{SchmidtDensity} and was made explicit by Cook and the author \cite{CookMagyarPrimes} where it was used to study prime solutions of diophantine systems of a fixed degree. This was extended by Yamagishi \cite{YamagishiPrimesDifferingDegrees} for systems of forms of differing degrees. However, as we need to stay in the framework of odd degree systems, we will state the regularization lemma accordingly.
	
	Let recall the definition of the rational Schmidt rank (also referred to as $h$-invariant or  strength). Let $f \in \mathbb{Q}[x_1, \ldots, x_s]$ be a form of degree $\ell \ge 2$. The Schmidt rank $h(f)$ is defined as the smallest non-negative integer $h$ such that $f$ can be written identically as
	\[ f = \sum_{i=1}^h u_i v_i \]
	where $u_i, v_i$ are forms of strictly positive degree in $\mathbb{Q}[x_1, \ldots, x_s]$. For a system of forms $\mathcal{F} = (f_1, \ldots, f_r)$ all of the same degree $\ell \ge 2$, the $h$-invariant is defined as the minimum rank of any non-trivial rational linear combination of the forms:
	\[ h(\mathcal{F}) = \min_{\mathbf{c} \in \mathbb{Q}^r \setminus \{\mathbf{0}\}} h\left( c_1 f_1 + \dots + c_r f_r \right). \]
	For a linear form $f(x)=a_1x_1+\ldots +a_sx_s$ we define $h(f)$ as the number of non-zero coefficients $a_i$. For a mixed-degree system $\FF=(\FF^{(d)},\ldots,\FF^{(1)})$ we define 
	\[
	h(\FF):=\min_{1\leq l\leq d} h(\FF^{(l)}).
	\]
	
	
	\begin{prop}[Odd-Degree Regularity Lemma] \label{prop:regularity}
		Let $d>1$ be an odd integer, and let $\mathcal{H}$ be a collection of functions $\mathcal{H}_{i}:\mathbb{Z}_{\ge 0}^2 \rightarrow \mathbb{Z}_{\ge 0}$, for $3 \le i \le d$, which are non-decreasing in each variable. For a collection of non-negative integers $r_{1}, \ldots, r_{d}$, there exist constants $C_{1}, \ldots, C_{d}$ (depending only on the functions $\mathcal{H}_i$ and the initial parameters $r_i$) such that the following holds.
		
		Given a system of forms $\mathcal{F}$ in $\mathbb{Z}[x_{1}, \ldots, x_{s}]$ consisting strictly of odd-degree forms bounded by $d$, let $r_\ell$ be the number of degree $\ell$ forms in $\mathcal{F}$. There exists an auxiliary system of forms $\mathcal{G}$ in $\mathbb{Q}[x_{1}, \ldots, x_{s}]$, also consisting strictly of odd-degree forms bounded by $d$, satisfying the following. For each $1 \le \ell \le d$, let $r_{\ell}^{\prime}$ be the number of degree $\ell$ forms in $\mathcal{G}$, let $\mathcal{G}^{(\ell)}$ denote the subsystem of degree $\ell$ forms, and let $R^{\prime}=r_{1}^{\prime}+ \cdots +r_{d}^{\prime}$.
		
		\begin{enumerate}
			\item \textbf{(Zero Locus Inclusion)} The common zero locus of the system $\mathcal{G}$ is contained in the common zero locus of $\mathcal{F}$. That is, $V_0(\mathcal{G}) \subseteq V_0(\mathcal{F})$. 
			
			\item \textbf{(Bounded Complexity)} For each $1 \le \ell \le d$, the number of forms $r_{\ell}^{\prime}$ is at most $C_{\ell}(r_{1}, \ldots, r_{d}, \mathcal{H})$.
			
			\item \textbf{(High Schmidt Rank)} For each $3 \le \ell \le d$, we have $h(\mathcal{G}^{(\ell)}) \ge \mathcal{H}_{\ell}(R^{\prime},d)$. Moreover, the linear forms of $\mathcal{G}^{(1)}$ are linearly independent over $\mathbb{Q}$.
		\end{enumerate}
	\end{prop}
	
	Proposition 3.1 in essence follows by arguments of Schmidt \cite[Theorem 1]{SchmidtDensity}, for the sake of completeness we sketch the proof in an appendix following \cite{CookMagyarPrimes}.
	
	Let us remark that we can also require that for each $3 \le \ell \le d$, there exists a constant $C_\ell(R,d,\HH)$ depending on the parameters $R,d$ and the growth functions $\HH$, such that $r^\prime\leq C_\ell(R,d,\HH)$, simply taking $C_\ell(R,d,\HH):=\max_{r_1 +\ldots r_d=R}C_\ell\prime(r_1,\ldots,r_d,d,\HH)$.  
	
	As originally observed by Schmidt \cite{SchmidtDensity}, this proposition is obtained by exploiting the recursive nature of the Schmidt rank. Any small rank an odd-degree form is decomposed as $f = \sum_{i=1}^h u_i v_i$, parity dictates that for every $i$, exactly one of the factors $u_i$ or $v_i$ must have an odd degree. By collecting only these odd-degree factors to generate the subsequent auxiliary systems, we guarantee both the pure odd-degree structure of $\mathcal{G}$ and the zero locus inclusion, see the Appendix.
	
	
	The second main result is a theorem providing an asymptotic for the number of prime solutions to diophantine systems of sufficiently high Birch rank. This has been obtained in \cite{CookMagyarPrimes} when all forms of the same degree and has been extended to mixed-degree systems by Yamagishi \cite{yamagishi2018prime}.

	Let $d>1$ and let
	\[
	\FF=(f_1,\ldots,f_d)
	\]
	be a system of polynomials in $\mathbb{Z}[x_1,\ldots,x_s]$, where for each $1< \ell\le d$ we write
	\[
	\FF^{(\ell)}=(f_{\ell,1},\ldots,f_{\ell,r_\ell})
	\]
	for the subsystem consisting of the degree-$\ell$ polynomials of $f$. 
	
	We recall that the Birch rank of the mixed-degree system $\FF$ is defined as 
	\[
	B(\FF):=\min_{1\leq l\leq d} B_l(\FF^{(\ell)})
	\]
	where the Birch rank of the degree-$\ell$ block $\FF^{(\ell)}$ is defined as 
	\[
	B_l(\FF^{(\ell)}) =codim\,\big\{\mathbf{z}\in \C^s: rank\,\bigg(\frac{\partial f_{\ell,i}}{\partial x_j}(\mathbf{z})\bigg) < r_l\big\},
	\]
	i.e. as the co-dimension of the singular locus $V^\ast (\FF^{(\ell)})$ of the homogeneous system $\FF^{(\ell)}$. We remark that for a single form $f$ we have that 
	$V^\ast (f)=\{\mathbf{z}\in \C^s: \nabla f(\uz)=0\}$. For $\ell=1$ we define $B_1(f)$ as above i.e. as the number of non-zero coefficients of the linear form $f$.
	
	Define the logarithmically weighted number of integer solutions $\ux\in [N]^s$ of the system $\FF(\ux)=0$,
	\[
	M_\FF(N):=\sum_{x\in[N]^s}\Lambda(x_1)\cdots\Lambda(x_s)\,\mathbf{1}_{\{\FF(x)=0\}},
	\]
	where $\Lambda$ is the von Mangoldt function.
	
	\begin{thm}[Yamagishi {\cite[Theorem~1.2]{yamagishi2018prime}}]
		There exists a function $\mathcal{R}_0(\ur,d)$ depending only on the variables $\ur=(r_1,\ldots,r_d)$ and $d$ such that the following holds. If
		$B_\ell(f_\ell)\geq \mathcal{R}_0(\ur,d)$ for all $1\leq \ell\leq d$ the there exist a constant $C(\FF)$ depending only on $\FF$ and a constant
		$c>0$ such that, as $X\to\infty$, such that
		\[
		M_\FF(N)=C(\FF)\,N^{\,s-D}+O\!\left(\frac{N^{\,s-D}}{(\log N)^c}\right),
		\]
		with $D:=\sum_{\ell=1}^d \ell\, r_\ell$.
		Moreover, $C(\FF)>0$ provided that the system has a non-singular real solution in $(0,1)^s$, 
		and for every prime $p$ the system also has a non-singular solution in $(\mathbb{Z}_p^\times)^s$, where $\Z_p^\times$ is the set of $p$-adic units.
	\end{thm}
	
	It follows by standard means \cite{yamagishi2018prime,liu2024forms} that under the same conditions the diophantine system $\FF(\ux)=0$ has $C(\FF)\,(N^{\,s-D}/(\log N)^s)\,(1+o(1))$ prime solutions $\ux\in (\PP\cap [N])^s$.
	
	Again, the function $R_0(\ur,d)$ can be replaced by a function $\RR_0(R,d)$ with $R=r_1+\ldots +r_d$, as we have remarked after the regularity decomposition. 
	
	At this point two main issues remains. One is that the odd-degree regularity decomposition guarantees that the auxiliary system $\mathcal{G}$ possesses a large rational Schmidt rank (strength). However, asymptotic results for prime solutions, such as Yamagishi's theorem above, are classically formulated under the assumption of a large Birch rank. To bridge this gap, one can appeal to a recent result of Baily and Lampert \cite{baily2024strength}, who established that the strength of a system of forms is bounded linearly by its Birch rank, namely the following. 
	
	\begin{thm}[Baily-Lampert{\cite[Lemma 2.12]{baily2024strength}}] Let $\FF=(f_1,\ldots,f_R)$ be a system of rational forms of degree at most $d$. Then one has 
		\[
		h(\FF)\leq C(d)R\,(B(\FF)+R-1)
		\]
		where one may take $C(d)=4^d d^2$
	\end{thm}
	
	The second main difficulty is to show the existence of non-singular solutions $\Z_p^*$ for all primes $p$ which will be the subject of Section 5.
	
	\section{Regularity and the exceptional set of zeroes.}
	
	Let $\HH_\ell(R,d)$ for $1\leq \ell\leq d$ ($\ell$ odd) be a family of growth functions on $\Z_{\geq 0}^2$. We will not yet specify them but rather we will derive some lower bounds on them which can be expressed in terms of the initial parameters of the system $\FF$. We also choose a function $\mathcal{\RR}(R,d)$ which is sufficiently larger than the function $\mathcal{\RR}_0(R,d)$ in Theorem 3.1, so that all our systems after the regularization and various further manipulations, will have Birch rank at least $\mathcal{\RR}_0(R,d)$.
	
	Suppose we apply Proposition 3.1 to the initial system $\FF$ consisting of $R$ forms of odd degrees at most $d$. The we get a new system $\GG$ of $R^\prime$ of forms of odd-degree at most $d$. If there would be non-singular solutions in $(\Z_p^\times)^s$ for all $p\in \PP$ and if we would also have that $h(\GG^{(1)})=B_1(\GG^{(1)})\geq \RR(R'd)$ the we could apply Theorem 3.1 which would yield many prime solutions. 
	
	However, we only know that the linear forms in $\GG^{(1)}=(l_1,\ldots,l_{r_1^\prime})$ are linearly independent so $B_1(\GG^{(1)})\geq 1$. We first address this issue. Suppose there exists a non-trivial linear combination $l:=\la_1 l_1+\ldots +\la_{r_1^\prime} l_{r_1^\prime}$ has rank less that $\RR(R'd)$. The we set all variables $x_i$ corresponding to a non-zero coefficient of the linear form $l$ to be zero. We restrict the system $\GG$ to the subspace where all such coordinates $x_i=0$. On this subspace $l=0$ hence one of the linear forms, say $l_{r_1^\prime}$ is a linear combination of the remaining linear forms and can be omitted from the system without increasing the common zero set of the system $V_\GG$. Note that the Schmidt rank of the system decreases by at most 1, when the system is restricted to a coordinate hyperplane, say $x_1=0$. Indeed any form $f(x_1,\ldots,x_s)$ can be written as $f(x_1,\ldots,x_s)=x_1 g(x_1,\ldots,x_s)+h(x_2,\ldots,x_s)$. Thus the restricted system $\GG^\prime$ will satisfy $h(\GG'^{(l)})\geq H_l(R',d)-\RR(R',d)$ for all $3\leq l\leq d$. We repeat the procedure if there is still a linear combination of the remaining linear forms that has rank less than $\RR(R',d)$; until we have removed all small rank linear combinations of the system $\GG$.
	The we arrive to a system $\GG^\prime$ defined on $s'\geq s-\RR(R',d)R'$ variables such that 
	\begin{equation}\label{4.1}
		B_1(\GG'^{(1)})\geq \RR(R',d) 
	\end{equation}
	where $R'\leq R$ is the number of forms of the system $\GG^\prime$.
	
	We also do some preparations for the second and more substantial issue, that is to find non-singular $p$-adic solutions with non-vanishing coordinates. Note that the solution set $V_\GG^\prime$ has co-dimension $R'$ over $\C$ as it has a small dimensional singular locus. Initially, we may have that some maximal $R'$-codimensional irreducible components that are contained in some coordinate-hyperplane $H_1=\{x_{i_1}=0\}$. If this is the case set $x_i=0$ and restrict the variety to $V_{\GG^\prime}$ to $H_i$. Note that its co-dimension drops by 1. Then we look at the restricted variety $V_{\GG_1}$. If it has a max dimensional irreducible component still contained in a coordinate hyperplane $H_2$ restrict $V_{\GG_1}$ to that hyperplane obtaining $V_{\GG_2}$ in $H_2$ of co-dimension $R'-2$. We continue the procedure as long as this is the case, but we cannot have more that $R'$ steps, as after $R'$ steps the variety inside the ambient space will have co-dimension 0, so it cannot be contained in any coordinate hyperplanes. That means that setting at most $R'\leq C(d,R)$ coordinates to zero, none of the maximal dimensional irreducible components will be contained in a coordinate hyperplane. We will assume this from now on.
	
	Let $J_F\subs [s]$ be the set of all indices $i\in [s]$ such that we set $x_i=0$ during these two procedures. The set $J_\FF$ will be the exceptional set of zero coordinates in Theorem 2.2. Write $\ux=(\ux',\uz)$ where $\uz=(x_i)_{i\in J_\FF}$. If $g(\ux',0)=0$
	for all $g\in \GG^\prime$ then $g'(\ux'):=g(\ux',0)=0$ for all $g\in \GG$ including the linear forms in $\GG$, and hence $f(\ux',0)=0$ for all $f\in\FF$ for the initial system. Thus it will be enough to show that there exists a constants $Y_{\GG^\prime}>0$ and $C_{\GG^\prime}>0$ such that if $\NN^\ast_{\GG^\prime}(N,Y)$ denote the number of integer solutions $\ux'\in [-N,N]^{s'}$ to he system of equations $g'(\ux')=0$ for all $g'\in \GG^\prime$ such that every coordinate $x_i=y_ip_i$ with $1\leq y_i\leq Y$ and $p_i\in\PP$, then 
	\begin{equation}\label{4.2}
		\NN^\ast_{\GG^\prime}(N,Y) \geq C_{\GG^\prime}\ 
		\frac{N^{s'-D}}{(\log\,N)^{s'}},
	\end{equation}
	for some constants $Y=Y_{\GG^\prime}$ and $D=D(R',d)$. Thus, as $s'\geq s-\RR(R',d)R'\geq s-s(R,d)$ (as $R'\leq C(R,d)$ by Theorem 3.1)..
	
	We also have for $3\leq l\leq d$, 
	\[
	h(\GG'^{(\ell)})\geq H_l(R',d)- \RR(R',d)R'
	\]
	Thus by Theorem 3.2 we have that 
	\begin{equation}\label{4.3}
		B_l(\GG'^{(l)})\geq \HH_\ell'(R',d):=c(d)\big((R')^{-1} H_l(R',d)-\RR(R',d)-1\big),
	\end{equation}
	with $c(d)=4^{-d}d^{-2}$.
	
	Thus in effect we have reduced proving Theorem 2.2 to counting almost prime solutions of the regular system $\mathcal{G}$ consisting of $R' \le C(R, d)$ odd-degree forms in $s'$ variables. We ensure that our target rank threshold $\mathcal{R}(R', d)$ is chosen to be sufficiently large to satisfy not only the prerequisites of Yamagishi's global theorem, but also the local finite field rank conditions $B_0(R',d)$ that will be required in order to find non-singular solutions above $\F_p$, in Section 5. To formalize this, we state the following theorem for the system $\mathcal{G}'$. 
	
	\begin{thm}\label{thm:regularized_main}
		Let $\mathcal{G}'$ be the reduced system of $R'$ forms in $s'$ variables constructed above, satisfying the Birch rank conditions $B_\ell(\mathcal{G}'^{(\ell)}) \ge \mathcal{R}(R', d)$ for all $1 \le \ell \le d$. There exists a constant $Y_{\mathcal{G}'} > 0$ and strictly positive integers $y_i \le Y_{\mathcal{G}'}$ (for $1 \le i \le s'$) such that the lower bound \eqref{4.2} holds, i.e.
		\[
		\NN^\ast_{\GG^\prime}(N,Y) \geq C_{\GG^\prime}\ 
		\frac{N^{s'-D_{\GG^\prime}}}{(\log\,N)^{s'}}
		\]
		with $D_{\GG^\prime} = \sum_{\ell=1}^d \ell r'_\ell \le d R'$. 
	\end{thm}
	
	Before proceeding to the proof of Theorem \ref{thm:regularized_main} in Sections 5 and 6, we first demonstrate that it easily implies our main result, Theorem 2.2.
	
	\begin{proof}[Proof that Theorem \ref{thm:regularized_main} implies Theorem 2.2]
		Let $J_F \subset [s]$ be the exceptional set of indices corresponding to the variables eliminated during the linear independence and geometric ``cleanup" procedures above. By construction, $|J_F| \le \mathcal{R}(R', d) R' + R'$, which is a constant bounded in terms of the initial parameters $R$ and $d$. Let $I = [s] \setminus J_F$ denote the indices of the $s'$ surviving (active) variables.
		
		By setting the growth function $\HH_\ell(R,d)$ sufficiently large we can ensure via estimates \eqref{4.1} and \eqref{4.3} that the conditions of Theorem \ref{thm:regularized_main} hold for the restricted system $\mathcal{G}'$ on the variables indexed by $I$. For each valid solution $\mathbf{x}' \in [-N, N]^{s'}$ to $\mathcal{G}'(\mathbf{x}') = 0$, we construct a vector $\mathbf{x} \in \mathbb{Z}^s$ by embedding $\mathbf{x}'$ into the coordinates indexed by $I$, and setting $x_j = 0$ for all $j \in J_F$. 
		
		Because the zero locus of $\mathcal{G}'$ is contained in the zero locus of the original system $\mathcal{F}$ (as guaranteed by Proposition 3.1), and because the variables in $J_F$ were explicitly set to zero to define the restricted system, it follows immediately that $\mathcal{F}(\mathbf{x}) = 0$. 
		
		By Theorem \ref{thm:regularized_main}, there are at least $C_{\GG^\prime}\, N^{s' - D_{\GG^\prime}} / (\log N)^{s'}$ such solutions where $x_i = y_i p_i$ for $i \in I$. As the system $\GG^\prime$ depends only of the initial system $\FF$ the constant $C_{\GG^\prime}>0$ will depend only of the initial system. Padding these vectors with zeros does not change the bounds on the active coordinates. Since $s' \ge s - |J_F|$, we obtain at least
		\[
		\NN_{\FF,J_\FF}(N,Y_\FF)\gg_{\mathcal{F}} \frac{N^{s - |J_F| - D}}{(\log N)^{s - |J_F|}} \gg_{\mathcal{F}} \frac{N^{s - D'}}{(\log N)^s}
		\]
		almost-prime solutions to the original system $\mathcal{F}$, where $D' = D_{\GG^\prime} + |J_F|$ is a constant depending only on $R$ and $d$. This establishes Theorem 2.2.
	\end{proof}
	
	The remainder of this paper is dedicated to proving Theorem \ref{thm:regularized_main}.
	We will write $\GG:=\GG^\prime$, $R:=R'$ and $s:=s'$ for simplicity of notations.
	
	
	\section{Finding non-singular $p$-adic solutions.} If the system $\GG$ possessed non-singular solutions in $(\Z_p^\times)^s$ for all $p\in \PP$  and all of its degree-$l$ blocks have Birch rank $B_l(\GG^{(\ell)})\geq \RR(R,d)$ then we could apply Theorem 3.2 to find sufficiently many solutions already among the primes. We will first show that there exists a threshold $P_\GG$ such that this is indeed the case for all primes $p\geq P_\GG$. For the rest of the primes $p\leq P_\GG$ we first show that there exists a non-singular $\ux=(x_1,\ldots,x_s)$ in $\Z_p^s$ such that $x_i\neq 0$ for all $i\in [s]$.
	
	We start with the case of the large primes where one has good reduction $\pmod{p}$. It is a well-known fact in algebraic geometry that the dimension of a projective algebraic variety$V^\ast$, defined by a finite system $\GG$ integral homogeneous polynomials, is equal over $\C$ and over $\bar{\mathbb{F}}_p$, for all sufficiently large primes $p>P_\GG$. This follows from the generic flatness and constancy of the Hilbert polynomial in flat projective families, see \cite[III. Theorem 9.9]{HartshorneAlgebraicGeometry}. This implies that the Birch rank $B(\GG)$ of the system $\GG$, defined as the co-dimension of the singular locus over $\C$ is equal to the Birch rank of the system $\GG$ over $\bar{\F}_p$. 
	
	Let $p > P_{\mathcal{G}}$ be a fixed prime. Let $\mathcal{G} = (g_1, \dots, g_R)$ be considered as a map from $\mathbb{F}_p^s$ to $\mathbb{F}_p^R$, with $\mathcal{G}^{(\ell)}$ being the degree-$\ell$ block of the system. Over $\overline{\mathbb{F}}_p$, the Birch rank of each block satisfies $B_\ell(\mathcal{G}^{(\ell)}) \ge \mathcal{R}(R, d)$. By our construction in Section 4, $\mathcal{R}(R, d) \ge B_0(R, d)$, where $B_0(R, d)$ is the rank threshold required to dominate the error term in the finite field point count. Let $V_{\mathcal{G}} := \{\mathbf{x} \in \mathbb{F}_p^s : g_j(\mathbf{x}) = 0, \forall j \in [R]\}$. Because the Birch rank $B_l(\mathcal{G}^{(l)}) \ge B_l(\mathbf{b} \cdot \mathcal{G}^{(l)})$ for any vector $\mathbf{b} \neq 0 \in \bar{\F}_p^{R_l}$, it is enough to show the following lemma to guarantee that $V_{\mathcal{G}}$ possesses non-singular solutions.
	
	
	\begin{lem}\label{lem5.1}
		Let $R,d,s\geq 1$.  There are constants
		\[
		B_0=B_0(R,d)
		\qquad\text{and}\qquad
		p_0=p_0(s,d)
		\]
		with the following property.  Let $p>p_0$, and let
		\[
		\mathcal G=(g_1,\ldots,g_R):\mathbb F_p^s\longrightarrow \mathbb F_p^R
		\]
		be a system of homogeneous forms of degrees at most $d$.  For
		$1\leq l\leq d$, let $\mathcal G^{(l)}$ denote the sub-system
		consisting of those $g_i$ which have degree $l$, and write
		$R_l=|\mathcal G^{(l)}|$.  Assume that each non-empty block
		$\mathcal G^{(l)}$ has Birch rank at least $B_0$ over
		$\overline{\mathbb F}_p$, in the sense that for every
		$0\neq \mathbf b\in\overline{\mathbb F}_p^{R_l}$ the singular locus of $\ \mathbf b\cdot \mathcal G^{(\ell)}\ $ has co-dimension at least $B_0$ in $\mathbb A^s_{\overline{\mathbb F}_p}$.
		Then, one has
		\[
		\bigl||V_{\mathcal G}|-p^{s-R}\bigr|\leq p^{s-R-1}.
		\]
	\end{lem}
	
	\begin{proof}
		We use the standard Birch--Weyl exponential-sum estimate in the form
		proved in Proposition~2 of \cite{CookMagyarRestrictedAP}.  More precisely,
		the argument there implies the following.  If
		$F:\mathbb F_p^s\to\mathbb F_p$ has degree $1\leq l\leq d$ and if its
		leading homogeneous part has Birch rank at least $B$ over
		$\overline{\mathbb F}_p$, then
		\begin{equation}\label{eq:birch-weyl-mixed-degree}
			\left|
			p^{-s}\sum_{x\in\mathbb F_p^s} e_p(F(x))
			\right|
			\ll_{s,d} p^{-c_d B},
		\end{equation}
		where $e_p(t)=e^{2\pi i t/p}$ and $c_d>0$ depends only on $d$.
		Indeed, after $l-1$ Weyl differencings all lower-degree terms vanish, and one is left with the multi-linear forms associated with the leading
		homogeneous part; the exceptional set is then bounded in terms of the
		singular locus exactly as in the proof of Proposition~2 of
		\cite{CookMagyarRestrictedAP}.
		
		Then,
		\[
		|V_{\mathcal G}|
		=
		p^{-R}\sum_{\mathbf a\in\mathbb F_p^R}
		\sum_{x\in\mathbb F_p^s}
		e_p(\mathbf a\cdot\mathcal G(x)).
		\]
		The term $\mathbf a=0$ gives $p^{s-R}$.  Let $\mathbf a\neq 0$, and
		decompose it according to degrees,
		\[
		\mathbf a\cdot\mathcal G
		=
		\mathbf a_d\cdot\mathcal G^{(d)}
		+\cdots+
		\mathbf a_1\cdot\mathcal G^{(1)} .
		\]
		Let $\ell$ be the largest index for which $\mathbf a_\ell\neq 0$.  Then
		the leading homogeneous part of $\mathbf a\cdot\mathcal G$ is
		\[
		\mathbf a_l\cdot\mathcal G^{(l)}.
		\]
		Since $\mathbf a_l\neq 0$ also as a vector over
		$\overline{\mathbb F}_p$, the rank hypothesis gives Birch rank at least
		$B_0$ for this leading form.  Hence \eqref{eq:birch-weyl-mixed-degree}
		gives
		\[
		\left|
		p^{-s}\sum_{x\in\mathbb F_p^s}
		e_p(\mathbf a\cdot\mathcal G(x))
		\right|
		\ll_{s,d} p^{-c_dB_0}.
		\]
		Choosing $B_0=B_0(R,d)$ sufficiently large, and then $p_0=p_0(s,d)$
		sufficiently large, we may ensure that the last quantity is at most
		$p^{-R-1}$ for every $\mathbf a\neq 0$.  Therefore
		\[
		\begin{aligned}
			\bigl||V_{\mathcal G}|-p^{s-R}\bigr|
			&\leq
			p^{s-R}
			\sum_{\mathbf a\in\mathbb F_p^R\setminus\{0\}}
			\left|
			p^{-s}\sum_{x\in\mathbb F_p^s}
			e_p(\mathbf a\cdot\mathcal G(x))
			\right|  \\
			&\leq
			p^{s-R}\,p^R\,p^{-R-1}
			\leq p^{s-R-1}.
		\end{aligned}
		\]
		This proves the lemma.
	\end{proof}
	
	It is well-known and easy to see that the Birch rank of a form drops by at most 2 when restricted to a coordinate hyperplane ${x_i=0}$. Thus the number of solutions in $\F_p^s$ with one coordinate equal to zero is $\ll_s p^{s-R-1}$. Thus for $p>P_\GG$ sufficiently large, we have that $V_G$ contains at least $\frac{1}{2} p^{s-R}$ points in $(\F_p^\ast)^s$. 
	
	Finally, we need to remove the singular points from our reduced solution set. By our assumption the singular locus $V_{\GG^{(\ell)}}^\ast$ of each $l$-degree block has dimension at most $s-B$. We need to estimate though the dimension of the singular locus of the whole mixed-degree $\GG$ system over $\overline{\F}_p$. Note that this is not the Birch rank of the system which is defined as the minimum of the Birch rank of it homogeneous degree-$\ell$ systems $\GG^{(\ell)}$. However for $p>d$ i.e. when the characteristic of the base field is bigger than $d$, we can use the following consequence of a``rank-drop" lemma proved by Lampert \cite{LampertDensity}.  
	
	\begin{lem}[Lampert{\cite[Lemma~4.5]{LampertDensity}}] If
		$\mathcal G=(\mathcal G^{(1)},\ldots,\mathcal G^{(d)})$ is a mixed-degree
		system with $R=\sum_\ell R_\ell$ equations and if each non-empty degree
		block has Birch rank at least $B$ over the algebraic closure, then the
		singular locus
		\[
		V^\ast_{\mathcal G}
		:=
		\{\ux\in \overline{\F}_p^s:\operatorname{rank}\operatorname{Jac}_{\mathcal G}(\ux)<R\}
		\]
		satisfies
		\[
		\operatorname{codim} V^\ast_{\mathcal G}
		\geq B-R-d+2.
		\]
	\end{lem}
	
	Since the singular locus is defined by $\binom{s}{R}$ minors, each being a polynomial of degree at most $R(d-1)$ the degree of the singular locus satisfies $deg (V^\ast_\GG)\ll_{s,d,R} 1$ thus
	\[
	\big| V^\ast_\GG \cap \F_p^s\big| \ll_{s,d,R} p^{s+R+d-B-2}\leq p^{s-R-1},
	\]
	for sufficiently large $p>P_\GG$ (in particular for $p>P_{s,d,R})$. Thus one can remove all singular points from the reduced solution set $V_\GG \cap (\F_p^\ast)^s$. Thus we have non-singular solutions which have no zero coordinates $\pmod{p}$ and then by a standard application of Hensel's lemma \cite{GreenbergForms} one obtains non-singular $p$-adic solutions in $(\Z_p^\times)^s$.
	
	
	Next, we address the set $S$ of small primes $q< P_\GG$. In this section we will show that there are non-singular $p$-adic solutions, which may or may not be units in $\Z_p^s$ but have non-zero coordinates. For the finitely many small primes we use the general local solubility
	theorem of Wooley \cite{WooleyLocalSolubility}. It states that for fixed $R$ and $d$, any system of $R$
	homogeneous forms of degree at most $d$ over $\mathbb Q_p$ has a non-trivial
	$p$-adic zero, provided the number of variables $s\geq\ga(R,d)$ is sufficiently large with respect to $R$ and $d$. To find non-singular $p$-adic solutions we shall use the following variant of Brandes' argument
	\cite[Lemma~4.1]{BrandesPadicSolubility}. 
	
	
	\begin{lem}\label{lem5.3}
		Let $\mathcal{G} = (g_1, \dots, g_R)$ be a system of homogeneous forms of odd degrees at most $d$ defined over $\mathbb{Q}$ in $s$ variables. Assume that the co-dimension of the singular locus satisfies
		\[
		s - \dim V^*_{\mathcal{G}} \ge \gamma(R, d)
		\]
		and that no geometrically irreducible component of maximal dimension of the variety $V_{\mathcal{G}}$ is contained in any coordinate hyperplane $H_i = \{x_i = 0\}$. Then there exists $\mathbf{x} \in \mathbb{Z}_p^s$ such that:
		\begin{enumerate}
			\item $\mathcal{G}(\mathbf{x}) = 0$,
			\item $x_i \neq 0$ for all $1 \le i \le s$,
			\item $\mathbf{x}$ is a non-singular $p$-adic solution to the system $\mathcal{G}$.
		\end{enumerate}
	\end{lem}
	
	\begin{proof}
		We apply the elementary Bertini-type consequence (as in Brandes~\cite[Lemma~4.1]{BrandesPadicSolubility} or Marmon~\cite[Lemma~2.6]{MarmonDensity}). Because the co-dimension of the singular locus $V^*_{\mathcal{G}}$ is at least $\gamma(R, d)$, there is a linear subspace \(M\subset \mathbb A^s_{\mathbb Q_p}\) defined over $\mathbb{Q}_p$ which meets $V^*_{\mathcal{G}}$ only at the origin. By Wooley's theorem \cite{WooleyLocalSolubility}, $M$ contains a non-trivial $p$-adic zero $\ux^* \in \mathbb{Q}_p^s$. Because $\ux^*$ avoids the singular locus, it is a non-singular solution.
		
		By the \(p\)-adic implicit function theorem \cite[Chapter 3]{GreenbergForms}, since \(\ux^\ast\) is a
		non-singular zero of \(G\), the set of \(p\)-adic zeros of \(G\) in a
		sufficiently small neighborhood of \(\ux^\ast\) is a \(p\)-adic analytic
		manifold \(U\) of dimension \(s-R\).
		
		We claim that \(U\) is not contained in any coordinate hyperplane
		\(H_i=\{x_i=0\}\).  Suppose, for contradiction, that
		\[
		U\subset H_i .
		\]
		Then the coordinate function \(x_i\) vanishes on a non-empty \(p\)-adic
		open subset of the smooth local branch of \(V_G\) through \(\ux^\ast\).
		By the \(p\)-adic analytic identity theorem (cf. Serre~\cite[Part~II, Chapter~III]{SerreLie}), \(x_i\) vanishes identically
		on this local branch. 
		
		Since \(\ux^\ast\)  is a smooth point of \(V_{\GG}\), this local branch is
		Zariski dense in the unique irreducible component of
		\(V_{\GG}\otimes_{\mathbb Q}\overline{\mathbb Q}_p\) passing through
		\(\ux^\ast\). This component has maximal dimension. Hence
		\(V_{\GG}\otimes_{\mathbb Q}\overline{\mathbb Q}_p\) has a
		maximal-dimensional irreducible component contained in \(H_i\). This contradicts the coordinate cleanup carried out in
		Section~4, which ensured that no maximal-dimensional irreducible component of $V_\GG$ is contained in a coordinate hyperplane. Indeed, though the cleanup was proved over $\C$, equivalently over
		$\overline{\mathbb{Q}}$, but since $V_\GG$ is defined over $\mathbb{Q}$, the resulting geometric statement; namely that no maximal-dimensional irreducible component is contained in a coordinate hyperplane, remains valid after base change to $\overline{\mathbb{Q}}_p$, see, for example Shafarevich~\cite[Chapter I, Section 6]{ShafarevichBasicAG1}.

		Hence \(U\not\subset H_i\) for every \(i\).  It follows that each
		\(U\cap H_i\) is a proper \(p\)-adic analytic subset of \(U\), and so
		\[
		U\setminus \bigcup_{i=1}^s (U\cap H_i)
		\]
		is non-empty.  Choosing
		\[
		\ux^{\ast\ast}\in U\setminus \bigcup_{i=1}^s H_i,
		\]
		we obtain a non-singular \(p\)-adic solution of \(G(x)=0\) with
		\[
		x^{\ast\ast}_i\neq 0
		\qquad (1\leq i\leq s).
		\]
		Finally, since \(G\) is homogeneous, multiplying \(\ux^{\ast\ast}\) by a
		sufficiently large power of \(p\) gives a solution in \(\mathbb Z_p^s\)
		which is still non-singular and has no zero coordinate.
	\end{proof}
	
	
	In order to apply Theorem 3.1, we also need to guarantee a non-singular real solution $\mathbf{x}=(x_1,\ldots,x_s)$ such that $x_i \neq 0$ for all $1 \le i \le s$. The proof of this proceeds similarly to the above lemma, but is somewhat simpler. 
	
	Let $\mathbf{x}^{(0)}\in (-1,1)^s$ be a non-singular real zero supplied by Schmidt's theorem. By the Implicit Function Theorem, the real zero set contains, in a sufficiently small neighborhood of $\mathbf{x}^{(0)}$, is a real analytic manifold of dimension $s-R$. If all points of this local manifold were contained in the union of the coordinate hyperplanes, then one of the coordinate functions $x_i$ would vanish identically on a nonempty open subset of this manifold, and hence on the corresponding complex analytic germ. Passing to the Zariski closure, this would force a maximal-dimensional irreducible component of the complex variety to be contained entirely in the coordinate hyperplane $x_i=0$, contradicting the geometric preparation in Section~4. 
	
	Therefore, by homogeneity, there exists a non-singular real solution $\mathbf{x}^* \in (-1, 1)^s$ such that $x_i^* \neq 0$ for all $1 \le i \le s$.
	
	\section{Scaling the system: Finding almost prime solutions}
	\label{sec:scaling}
	
	With the local geometric obstructions resolved, we in position to find almost prime solutions of the system $\GG$. Our strategy is to use the non-singular $p$-adic solutions from Lemma \ref{lem5.3} to construct a set of integer multipliers $y_i$. By scaling the variables of our original system $\mathcal{G}$ by these multipliers, we will force the system to have non-singular unit solutions at every prime, satisfying the local prerequisites for prime solubility as in Theorem 3.1.
	
	Let $S$ denote the set of ``bad'' primes as before, and note that such primes are bounded by a threshold $P_\GG$, depending on the system $\GG$. For each $p \in S$, Lemma \ref{lem5.3} guarantees the existence of a $p$-adic integer point $\mathbf{x}^{(p)} \in \mathbb{Z}_p^s$ such that $\mathcal{G}(\mathbf{x}^{(p)}) = 0$, the Jacobian matrix evaluated at $\mathbf{x}^{(p)}$ has maximal rank, and $x_i^{(p)} \neq 0$ for all $1 \le i \le s$.
	
	Because $x_i^{(p)} \neq 0$, each coordinate has a well-defined, finite $p$-adic valuation. We may therefore write
	\[
	x_i^{(p)} = p^{e_{i,p}} u_i^{(p)}
	\]
	where $e_{i,p} \ge 0$ is a non-negative integer and $u_i^{(p)} \in \mathbb{Z}_p^\times$ is a $p$-adic unit. We define the global scaling multipliers $y_i$ for $1 \le i \le s$ by taking the product over the bad primes:
	\[
	y_i = \prod_{p \in S} p^{e_{i,p}}.
	\]
	By construction, each $y_i$ is a strictly positive rational integer whose prime factors lie exclusively in $S$. We now define the scaled system of forms $\mathcal{G}_{\mathbf{y}}$ by the linear change of variables
	\[
	\mathcal{G}_{\mathbf{y}}(\mathbf{w}) = \mathcal{G}(y_1 w_1, \dots, y_s w_s).
	\]
	
	\begin{lem}\label{lem:scaled_system_properties}
		Let $\mathcal{G}_{\mathbf{y}}$ be the scaled system of forms defined above. Then $\mathcal{G}_{\mathbf{y}}$ satisfies the following properties:
		\begin{enumerate}
			\item The Birch rank of $\mathcal{G}_{\mathbf{y}}$ is equal to the Birch rank of $\mathcal{G}$.
			\item For every prime $p$, the system $\mathcal{G}_{\mathbf{y}}$ has a non-singular $p$-adic solution $w^{(p)} \in (\mathbb{Z}_p^\times)^s$ consisting entirely of $p$-adic units.
		\end{enumerate}
	\end{lem}
	
	\begin{proof}
		First, we consider the Birch rank. The transformation $\mathbf{w} \mapsto (y_1 w_1, \dots, y_s w_s)$ is a non-degenerate linear change of variables over $\mathbb{Q}$ and $\mathbb{C}$, as each $y_i$ is a positive integer. Such a transformation preserves the dimension of the variety and the co-dimension of its singular locus. Hence, the Birch rank of $\mathcal{G}_{\mathbf{y}}$ is identical to that of $\mathcal{G}$.
		
		Next, we establish the existence of local unit solutions. We divide the primes into two cases based on the exceptional set $S$.
		
		\textbf{Case 1: Bad primes ($p \in S$).} We define the coordinates $w_i^{(p)} \in \mathbb{Z}_p$ so that
		\[
		x_i^{(p)} = y_i w_i^{(p)} = p^{e_{i,p}} M_{i,p} w_i^{(p)},
		\]
		where $M_{i,p} = \prod_{q \in S, q \neq p} q^{e_{i,q}}$. Because $p$ does not divide $M_{i,p}$, it is a unit in $\mathbb{Z}_p$. Since $x_i^{(p)} = p^{e_{i,p}} u_i^{(p)}$ by construction, we obtain $w_i^{(p)} = u_i^{(p)} M_{i,p}^{-1}$, which is indeed a unit in $\mathbb{Z}_p$. Now,
		\[
		\mathcal{G}_{\mathbf{y}}(w^{(p)}) = \mathcal{G}(y_1 w_1^{(p)}, \dots, y_s w_s^{(p)}) = \mathcal{G}(x_1^{(p)}, \dots, x_s^{(p)}) = 0,
		\]
		hence $w^{(p)}$ is a solution. Furthermore, because any $R \times R$ minor of the Jacobian matrix of $\mathcal{G}_{\mathbf{y}}$ evaluated at $w^{(p)}$ is simply the corresponding minor of $\mathcal{G}$ at $x^{(p)}$ scaled by a product of the non-zero constants $y_i$, the new solution $w^{(p)}$ must also be non-singular.
		
		\textbf{Case 2: Good primes ($q \notin S$).} For $q \notin S$, the system $\mathcal{G}$ we have already shown that to have a non-singular $q$-adic unit solution $x^{(q)} \in (\mathbb{Z}_q^\times)^s$. Furthermore, because $q \notin S$, the prime $q$ does not divide $y_i$ for any $1 \le i \le s$. Therefore, each multiplier is a $q$-adic unit ($y_i \in \mathbb{Z}_q^\times$). We define $w^{(q)} \in \mathbb{Z}_q^s$ by $w_i^{(q)} = x_i^{(q)} y_i^{-1}$. Since both $x_i^{(q)}$ and $y_i^{-1}$ are units in $\mathbb{Z}_q$, their product $w_i^{(q)}$ is also a unit in $\mathbb{Z}_q$. Substituting this into the scaled system yields
		\[
		\mathcal{G}_{\mathbf{y}}(w^{(q)}) = \mathcal{G}(y_1 w_1^{(q)}, \dots, y_s w_s^{(q)}) = \mathcal{G}(x_1^{(q)}, \dots, x_s^{(q)}) = 0.
		\]
		By the same Jacobian scaling argument as above, $w^{(q)}$ is a non-singular solution.
		
		Thus, for every prime $p$, $\mathcal{G}_{\mathbf{y}}$ possesses a non-singular unit solution, completing the proof.
	\end{proof}
	
	At this point we are in the position to apply Yamagishi's theorem to the scaled system $\GG_y$. Which will imply Theorem 4.1 and hence our main result Theorem 2.2.
	
	We have already shown the existence of a real non-singular solution $\ux\in (-1,1)^s$ with no coordinates $x_i$ equal to 0. By choosing the signs of the scaling factors $y_i$ the same as signs of the $x_i$'s the scaled system $\GG_y$ will have a non-singular solution $\mathbf{w}\in (0,1)^s$.
	
	\subsection*{Proof of Theorem 4.1}
	
	By Lemma \ref{lem:scaled_system_properties}, the system of integer forms $\mathcal{G}_{\mathbf{y}}$ has sufficiently large Birch rank and possess non-singular $p$-adic unit solutions all every primes and a non-sngualr real solution in $(0,1)^s$. Therefore, the system perfectly satisfies the conditions of Yamagishi's theorem, that is those of Theorem 3.1.
	
	Let $Y = \max_{i \in I} y_i$. Note that $Y$ is a bounded constant depending only on  the original forms of the system $\GG$ and the exceptional set $S$. We apply Yamagishi's theorem to count prime solutions $\mathbf{q}$ to $\mathcal{G}_{\mathbf{y}}(\mathbf{q}) = 0$ within the box $[1, N/Y]^{|I|}$. The asymptotic formula guarantees that for $N$ sufficiently large, the number of such prime solutions is bounded below by
	\[
	\gg_{\GG}  \frac{(N/Y)^{|I| - D_\GG}}{(\log(N/Y))^{|I|}} \gg_{\GG} \frac{N^{|I| - D_\GG}}{(\log N)^{|I|}},
	\]
	with $D_\GG=\sum_{l=1}^d lr_l$ being the total degree of the system $\GG$.
	Recall from our initial geometric regularization that the number of active variables is $|I| \ge s - s(d,R)$, $s(d,R)$ being a quantity depending only the initial parameters $d$ and $R$. We define $D = s(d,R) + D_\GG$. Setting $x_i = y_i q_i$ for $i \in I$, and $x_i = 0$ for $i \notin I$, each prime solution $\mathbf{q}$ corresponds to a distinct almost-prime solution $\mathbf{x}$ to our original system $\mathcal{G}(\mathbf{x}) = 0$. 
	
	Because $q_i \le N/y_i$, it follows that $x_i \le N$. Thus, all constructed solutions $\mathbf{x}$ lie within the box $[-N, N]^s$. Since the multipliers $y_i \le Y$ are strictly bounded, we obtain at least
	\[
	c_{\GG}\ \frac{N^{s - D}}{(\log N)^s},\quad\quad (with\ some\ c_\GG>0)
	\]
	almost-prime solutions to the system, completing the proof of Theorem 4.1. \qed

	\section{Appendix: The odd degree regularity lemma}
	
	Here we give the proof of Proposition 3.1. The proof is carried out via a double induction on the parameters $d$ and $r_d$. First, we show that assuming the conclusions of the Proposition hold for $r_d$ with any choice of $r_{d-1}, \ldots, r_1$, this implies the result for $r_d+1$. Then the induction on $d$ is carried out in the same way. 
	
	For $d=1$, the Proposition follows by simply choosing a maximal linearly independent set of the linear forms. For $d=3$ and $r_3=1$, i.e., for a system $\FF=(f_3,\FF^{(1)})$, if $h(f_3) \geq \HH_3(R,3)$ with $R=r_3+1$, then we may take $\GG=\FF$. Otherwise, we have $f_3 = \sum_{i=1}^l u_i v_i$ where $l < \HH_3(R,3)$ for some forms $u_i, v_i$, where we may assume by parity that all the forms $v_i$ are linear. We may then adjoin the forms $v_i$ to the system $\FF^{(1)}$ to obtain a system $\GG^{(1)}$, and let $\GG$ be a maximal linearly independent set of the linear forms of $\GG^{(1)}$. It is clear that $V_\GG \subs V_\FF$.
	
	Now, for a fixed value of $d$, assume that the result holds for all systems with maximal degree $d$ for any given collection of growth functions $\HH$ when $r_d=j$ and $r_{d-2}, \ldots, r_1$ are arbitrary. Consider a fixed collection of odd-degree forms $\FF=(\FF^{(d)}, \ldots, \FF^{(1)})$ with $r_d=j+1$. Let $\FF'=(\FF^{(d-2)}, \ldots, \FF^{(1)})$. By the induction hypothesis, there is an odd-degree system $\GG^\prime$ which is a regularization of $\FF'$ with respect to $\HH_i'(R,d-2) := \HH_i(R,d-2)+j+1$ for all $1 \leq i \leq d-2$. 
	
	Now let $\tilde{\GG}'=(\FF^{(d)},\GG^\prime)$. If $\tilde{\GG}'$ fails to be a regularization of $\FF$, then $h(\FF^{(d)}) < \HH_d(R'+j+1,d)$, where $R'$ is the number of forms in $\GG^\prime$. Thus, there exist forms $u_i, v_i$ for $1 \leq i \leq l$, with $l < \HH_d(R'+j+1,d)$, such that
	\[
	\la_1 f_{1,d} + \ldots + \la_{j+1}f_{j+1,d} = \sum_{i=1}^l u_i v_i.
	\]
	Moreover, by parity, we may assume that all forms $v_i$ are of odd degree strictly less than $d$. Without loss of generality, we may assume $\la_{j+1} \neq 0$. Now let $\GG''$ be $\GG'$ adjoined with the forms $v_i$ which are not already linear combinations of the forms in $\GG^\prime$, and set
	\[
	\tilde{\GG}''=(f_{d,1}, \ldots, f_{d,j}, \GG'').
	\]
	Notice that if all the forms in $\tilde{\GG}''$ vanish at a point $\ux$, then all the forms vanish in $\FF'$ and also in $\FF^{(d)}$ at $\ux$; thus $V_{\tilde{\GG}''} \subs V_{\FF}$. By the induction hypothesis, there is a system $\GG$ which is a regularization of the system $\tilde{\GG}''$. It is clear from the construction that the number of forms in each $\GG^{(\ell)}$ is bounded by a constant depending only on $r_d, \ldots, r_1$. Thus, the odd-degree system $\GG$ is a regularization of $\FF$ satisfying conditions (1), (2), and (3). The induction step from $d$ to $d+1$ is carried out in the exact same way, taking $j=0$. \quad $\Box$
	
	\medskip
	
	\paragraph{\bf{Acknowledgements.}}
	The author thanks Amichai Lampert for helpful discussions.
	
	\bibliographystyle{plain}
	\bibliography{references}

\end{document}